\documentclass[reqno]{amsart}
\usepackage{amsmath,amsthm,amssymb,amscd}
\usepackage[shortlabels]{enumitem}
\usepackage{youngtab}
\usepackage{xcolor}

\newtheorem{theorem}{Theorem}[section]
\newtheorem{lemma}[theorem]{Lemma}

\newtheorem{corollary}[theorem]{Corollary}
\newtheorem{conj}[theorem]{Conjecture}
\theoremstyle{definition}

\newtheorem{remark}[theorem]{Remark}
\numberwithin{equation}{section}
\allowdisplaybreaks
\begin{document}

%
%
%
%
%
%
%
%
%

\author[Gurinder Singh]{Gurinder Singh (ORCID: 0000-0002-2319-2100)}
\address{Department of Mathematics, Indian Institute of Technology Guwahati, Assam, India, PIN- 781039}
\email{gurinder.singh@iitg.ac.in}

\title[Hook length biases in ordinary and $t$-regular partitions]
 {Hook length biases in ordinary and $t$-regular partitions}
 
\author[Rupam Barman]{Rupam Barman (ORCID: 0000-0002-4480-1788)}
\address{Department of Mathematics, Indian Institute of Technology Guwahati, Assam, India, PIN- 781039}
\email{rupam@iitg.ac.in}

\date{\today}





\begin{abstract} 
In this article, we study hook lengths of ordinary partitions and $t$-regular partitions. We establish hook length biases for the ordinary partitions and motivated by them we find a few interesting hook length biases in $2$-regular partitions. 
For a positive integer $k$, let $p_{(k)}(n)$ denote the number of hooks of length $k$ in all the partitions of $n$. We prove that $p_{(k)}(n)\geq p_{(k+1)}(n)$ for all $n\geq0$ and $n\ne k+1$; and $p_{(k)}(k+1)- p_{(k+1)}(k+1)=-1$ for $k\geq 2$. 
For integers $t\geq2$ and $k\geq1$, let $b_{t,k}(n)$ denote the number of hooks of length $k$ in all the $t$-regular partitions of $n$. We find generating functions of $b_{t,k}(n)$ for certain values of $t$ and $k$. 
Exploring hook length biases for $b_{t,k}(n)$, we observe that in certain cases biases are opposite to the biases for ordinary partitions. We prove that $b_{2,2}(n)\geq b_{2,1}(n)$ for all $n>4$, 
whereas $b_{2,2}(n)\geq b_{2,3}(n)$ for all $n\geq 0$. We also propose some conjectures on biases among $b_{t,k}(n)$.
\end{abstract}

\maketitle
\section{Introduction and statement of results} 
A partition of a positive integer $n$ is a finite sequence of non-increasing positive integers $\lambda=(\lambda_1, \lambda_2, \ldots, \lambda_r)$ such that $\lambda_1+\lambda_2+\cdots +\lambda_r=n$. 
Let $p(n)$ denote the number of partitions of $n$. By convention we take $p(0):=1$. The generating function for the partition function $p(n)$  is given by
\begin{align}\label{gen_fn_p(n)}
\sum_{n=0}^{\infty}p(n)q^n=\frac{1}{(q;q)_{\infty}}.
\end{align}
Here, $(a;q)_{\infty}:=\prod_{j=0}^{\infty}(1-aq^{j})$ is the \textit{$q$-Pochhammer symbol}. The number of parts of a partition $\lambda$ is called the \textit{length} of $\lambda$ and is denoted by $\ell(\lambda)$. Let $t\geq 2$ be a fixed positive integer. 
A $t$-regular partition of a positive integer $n$ is a partition of $n$ such that none of its parts is divisible by $t$.
\par A \textit{Young diagram} of a partition $(\lambda_1, \lambda_2, \ldots, \lambda_r)$ is a left-justified array of boxes with $i$-th row (from the top) having $\lambda_i$ boxes. 
For example, the Young diagram of the partition $(5,4,2,2,1)$ is shown in Figure \ref{Figure6.1} (left). The \textit{hook length} of a box in a Young diagram is the sum of the number of the boxes directly right to it, the number of boxes directly below it and 1 (for the box itself).
For example, see Figure \ref{Figure6.1} (right) for the hook lengths of each box in the Young diagram of the partition $(5,4,2,2,1)$.
\begin{figure}[h]
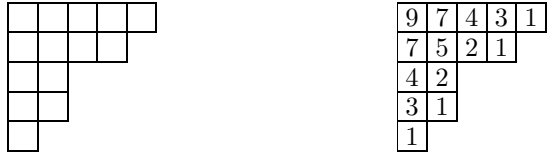
\label{Figure6.1}
\centering
\begin{minipage}[b]{0.4\textwidth}
\[\young(~~~~~,~~~~,~~,~~,~)\]
\end{minipage}
\begin{minipage}[b]{0.4\textwidth}
\[\young(97431,7521,42,31,1)\]
\end{minipage}
\caption{The Young diagram of the partition $(5,4,2,2,1)$ and its hook lengths}
\end{figure}
\par 
Hook lengths of partitions have important connection with representation theory of the symmetric groups $S_n$ and $GL_n(\mathbb{C})$. Hook lengths also appear in the Seiberg-Witten theory of random partitions, 
which gives Nekrasov-Okounkov formula for arbitrary powers of Euler's infinite product in terms of hook numbers. For more in these directions, see e.g. \cite{Garvan_1990, James, Littlewood,Nekrasov, Stanley}. 
Other than the ordinary partition function, hook lengths have also been studied for several restricted partition functions, for example, partitions into odd parts, partitions into distinct parts, partitions into odd and distinct parts, 
self conjugate partitions and doubled distinct partitions, see e.g. \cite{Ono_Singh, Ballantine_2023,Craig,Han_2016,Han_2017,Petreolle}. In this article, we study the total number of hooks of fixed length in $t$-regular partitions along with the 
number of hooks of fixed length in ordinary partitions. We establish hook length biases in ordinary partitions and motivated by them we find a few hook length biases for $t$-regular partitions.  
\par For a positive integer $k$, let $p_{(k)}(n)$ denote the number of hooks of length $k$ in all the partitions of $n$. Han \cite{Han_2010} derived a two-variable generating function for the number of partitions of $n$ with $m$ hooks of length $k$, say $p_{(k)}(m,n)$:
\begin{align}\label{6.0_Han}
\mathcal{F}_k(z;q):=\sum_{n,m\geq0}p_{(k)}(m,n)z^mq^n=\frac{(-(z-1)q^k;q^k)^k_{\infty}}{(q;q)_{\infty}}.
\end{align}
From \eqref{6.0_Han}, we obtain the generating function for $p_{(k)}(n)$:
\begin{align}\label{6.1_Han}
\sum_{n=0}^{\infty}p_{(k)}(n)q^n=\frac{\partial}{\partial z}\bigg\vert_{z=1}\mathcal{F}_{k}(z;q)=\frac{1}{(q;q)_{\infty}}\left(\frac{kq^k}{1-q^k}\right).
\end{align}
Table \ref{Table6.2} exhibits the values of $p_{(k)}(n)$ for some values of $k$ and $n$. Since the Young diagram of a partition $\lambda$ of $n\leq k-1$ can have at most $k-1$ boxes, there is no hook of length $k$ for $\lambda$. Hence, $p_{(k)}(n)=0$ for all $n$ satisfying $1\leq n\leq k-1$. 
Also, observe that the values in every column of Table \ref{Table6.2} are in non-increasing order, except from the last non-zero value in the column. 
That is, counting in all the partitions of $n$, the number of hooks of length $k$ is at least the number of hooks of length $k+1$, with one exception of $n=k+1$. We confirm this in the following theorem.
\begin{theorem}\label{Thm6.04}
For any positive integer $k$, $p_{(k)}(n)\geq p_{(k+1)}(n)$ for all $n\geq0$ and $n\ne k+1$. For $k\geq 2$, we have $p_{(k)}(k+1)-p_{(k+1)}(k+1)=-1$.
\end{theorem} 
For integers $t\geq2$ and $k\geq1$, let $b_{t,k}(n)$ denote the number of hooks of length $k$ in all the $t$-regular partitions of $n$. 
In \cite{Ballantine_2023}, Ballantine et al. studied hook lengths in partitions into odd parts, that is, in $2$-regular partitions. They derived the generating functions of $b_{2,2}(n)$ and $b_{2,3}(n)$, respectively:
\begin{align}
\sum_{n=0}^{\infty}b_{2,2}(n)q^n&=\frac{1}{(q;q^2)_{\infty}}\left(q^2+\sum_{n=2}^{\infty}\left(q^{2n-1}+q^{2(2n-1)}\right) \right),\label{6.0.1}\\
\sum_{n=0}^{\infty}b_{2,3}(n)q^n&=(-q^3;q)_{\infty}\frac{q^3(1+q^3)}{1-q^2}+(-q;q)_{\infty}\left(\frac{q^6}{1-q^4}+\frac{q^3}{1-q^6}\right).\label{6.0.2}
\end{align}
\begin{center}
\begin{table}
\caption{Values of $p_{(k)}(n): 1\leq n, k\leq 10$}\label{Table6.2}
\begin{tabular}[h]{c|cccccccccc}
$n\rightarrow$ & 1 & 2 & 3 & 4 & 5 & 6 & 7 & 8 & 9 & 10\\
\hline
$p_{(1)}(n)$ & 1 & 2 & 4 & 7 & 12 & 19 & 30 & 45 & 67 & 97\\
$p_{(2)}(n)$ & 0 & 2 & 2 & 6 & 8 & 16 & 22 & 38 & 52 & 82  \\
$p_{(3)}(n)$ & 0 & 0 & 3 & 3 & 6 & 12 & 18 & 27 & 45 & 63  \\
$p_{(4)}(n)$ & 0 & 0 & 0 & 4 & 4 & 8 & 12 & 24 & 32 & 52  \\
$p_{(5)}(n)$ & 0 & 0 & 0 & 0 & 5 & 5 & 10 & 15 & 25 & 40  \\
$p_{(6)}(n)$ & 0 & 0 & 0 & 0 & 0 & 6 & 6 & 12 & 18 & 30  \\
$p_{(7)}(n)$ & 0 & 0 & 0 & 0 & 0 & 0 & 7 & 7 & 14 & 21  \\
$p_{(8)}(n)$ & 0 & 0 & 0 & 0 & 0 & 0 & 0 & 8 & 8 & 16  \\
$p_{(9)}(n)$ & 0 & 0 & 0 & 0 & 0 & 0 & 0 & 0 & 9 & 9  \\
$p_{(10)}(n)$ & 0 & 0 & 0 & 0 & 0 & 0 & 0 & 0 & 0 & 10  
\end{tabular}
\end{table}
\end{center} 
Very recently, Craig et al. \cite{Craig} obtained the generating functions for $b_{2,k}(n)$, for any $k$. In this article, using the ideas of Ballantine et al. \cite{Ballantine_2023}, we obtain the generating functions of $b_{t,k}(n)$ for certain new values of $t$ and $k$. First, we find the generating function of $b_{t,1}(n)$ for all $t\geq 2$.
\begin{theorem}\label{Thm6.1}
For $t\geq2$, we have
\begin{align}\label{6.0.3}
\sum_{n=0}^{\infty}b_{t,1}(n)q^n=\frac{(q^t;q^t)_{\infty}}{(q;q)_{\infty}}\left(\frac{q}{1-q}-\frac{q^t}{1-q^t}\right).
\end{align}
\end{theorem}
In the following theorem, we find the generating functions of $b_{3,2}(n)$ and $b_{4,2}(n)$.
\begin{theorem}\label{Thm6.2}
We have
\begin{align}\label{6.0.4}
\sum_{n=0}^{\infty}b_{3,2}(n)q^n&=\frac{(q^3;q^3)_{\infty}}{(q;q)_{\infty}}\left(\frac{q^2}{1-q}+\frac{q^2}{1-q^2}-\frac{2q^3}{1-q^3}\right)
\end{align}
and
\begin{align}\label{6.0.5}
\sum_{n=0}^{\infty}b_{4,2}(n)q^n&=\frac{(q^4;q^4)_{\infty}}{(q;q)_{\infty}}\left(q^2+\sum_{\substack{n\geq2\\ 4\nmid n}}(q^n+q^{2n})-(q^4;q^8)_{\infty}\sum_{n=0}^{\infty}\frac{q^{8n+3}+q^{8n+5}}{1-q^{8n+4}}\right).
\end{align}
\end{theorem}
We next study the hook length biases among $2$-regular partitions. In Table \ref{Table6.1}, we observe that the values in each column, after the first entry, are in a non-increasing fashion except from the last non-zero value in the column. 
Unlike the case of ordinary partitions, in 2-regular partitions of $n$, the bias between hooks of length 1 and hooks of length 2 is the opposite. More precisely, we have the following theorem.
\begin{theorem}\label{Thm6.5}
We have $b_{2,2}(n)\geq b_{2,1}(n)$, for all $n>4$.
\end{theorem}
Theorem \ref{Thm6.5} motivates to study the biases among $b_{2,k}(n)$ or the difference $b_{2,k}(n)-b_{2,k+1}(n)$ for $k\geq2$. We find that the bias in Theorem \ref{Thm6.5} reverses for the next value of $k$. 
We prove that in 2-regular partitions of any non-negative integer $n$, the number of hooks of length 3 is no more than the number of hooks of length 2.
\begin{theorem}\label{Thm6.6}
We have $b_{2,2}(n)\geq b_{2,3}(n)$ for all $n\geq0$.
\end{theorem}
We believe that the bias for $b_{2,k}(n)$ continues for the next values of $k$ as obtained in Theorem \ref{Thm6.6}. 
More precisely, we conjecture the following.
\begin{center}
\begin{table}
\caption{Values of $b_{2,k}(n): 1\leq n, k\leq 10$}\label{Table6.1}
\begin{tabular}[h]{c|cccccccccc}
$n\rightarrow$ & 1 & 2 & 3 & 4 & 5 & 6 & 7 & 8 & 9 & 10\\
\hline
$b_{2,1}(n)$ & \textcolor{red}{1} & 1 & 2 & 3 & 4 & 6 & 8 & 11 & 14 & 19\\
$b_{2,2}(n)$ & 0 & \textcolor{red}{1} & 2 & 2 & 4 & 6 & 8 & 11 & 15 & 20  \\
$b_{2,3}(n)$ & 0 & 0 & \textcolor{violet}{2} & 1 & 2 & 5 & 5 & 7 & 11 & 15  \\
$b_{2,4}(n)$ & 0 & 0 & 0 & \textcolor{violet}{2} & 2 & 3 & 5 & 5 & 10 & 13  \\
$b_{2,5}(n)$ & 0 & 0 & 0 & 0 & \textcolor{orange}{3} & 1 & 3 & 5 & 6 & 10  \\
$b_{2,6}(n)$ & 0 & 0 & 0 & 0 & 0 & \textcolor{orange}{3} & 2 & 4 & 5 & 7  \\
$b_{2,7}(n)$ & 0 & 0 & 0 & 0 & 0 & 0 & \textcolor{magenta}{4} & 1 & 4 & 5  \\
$b_{2,8}(n)$ & 0 & 0 & 0 & 0 & 0 & 0 & 0 & \textcolor{magenta}{4} & 2 & 5  \\
$b_{2,9}(n)$ & 0 & 0 & 0 & 0 & 0 & 0 & 0 & 0 & \textcolor{blue}{5} & 1  \\
$b_{2,10}(n)$ & 0 & 0 & 0 & 0 & 0 & 0 & 0 & 0 & 0 & \textcolor{blue}{5}  
\end{tabular}
\end{table}
\end{center}
\begin{conj}\label{conj6.1}
For every integer $k\geq3$, $b_{2,k}(n)\geq b_{2,k+1}(n)$ for all $n\geq0$ and $n\ne k+1$. 
\end{conj}
\par Similar to the case of ordinary partitions, we notice certain patterns in the values in diagonal and upper diagonal of Table \ref{Table6.1}. In view of that, we prove the following result, 
which provides the exact values of $b_{2,k}(n)$, for $n=k,k+1$ (diagonal and upper diagonal entries in Table \ref{Table6.1}). 
\begin{theorem}\label{Thm6.4}
For every integer $k\geq1$, we have
\begin{align}
b_{2,k}(k)&=\left\{
\begin{array}{lll}
\frac{k+1}{2} & \text{if}\ k\equiv1\pmod2; \\
\frac{k}{2} & \text{if}\ k\equiv0\pmod2,	
\end{array}
\right.\label{6.0.41}\\
b_{2,k}(k+1)&=\left\{
\begin{array}{lll}
1 & \text{if}\ k\equiv1\pmod2; \label{6.0.42}\\
2 & \text{if}\ k\equiv0\pmod2.
\end{array}
\right.
\end{align}
\end{theorem}
From Theorem \ref{Thm6.4} we evaluate that for $k\geq1$
\begin{align}
b_{2,k}(k+1)-b_{2,k+1}(k+1)=\left\{
\begin{array}{lll}
-\frac{k-1}{2} & \text{if}\ k\equiv1\pmod2; \\
-\frac{k-2}{2} & \text{if}\ k\equiv0\pmod2,\label{6.0.01}	
\end{array}
\right.
\end{align}
Theorem \ref{Thm6.4} motivates us to extend \eqref{6.0.41} and \eqref{6.0.42} for any $t$-regular partitions. We find the exact values of $b_{t,k}(k)$ and $b_{t,k}(k+1)$ for any $t\geq3$ in the following result.
\begin{theorem}\label{Thm6.4a}
Let $t\geq3$ be an integer. Then 
\begin{enumerate}
\item for any positive integer $k\equiv r\pmod{t}$
\begin{align}
b_{t,k}(k)=\frac{(t-1)k+r}{t};\label{6.0.41a}
\end{align}
\item for any integer $k\geq4$
\begin{align}
b_{t,k}(k+1)&=\left\{
\begin{array}{lll}
\frac{(t-1)k+t}{t} & \text{if}\ k\equiv0\pmod{t}; \label{6.0.42a}\vspace{.12cm}\\
\frac{(t-1)k+r}{t} & \text{if}\ k\equiv r\pmod{t}\ \text{for}\ 1\leq r\leq t-2;\vspace{.12cm}\\
\frac{(t-1)k-1}{t} & \text{if}\ k\equiv t-1\pmod{t}.
\end{array}
\right.
\end{align}
Also, $b_{t,1}(2)=2$, for any $t\geq3$; $b_{3,2}(3)=1$; $b_{t,2}(3)=2$, for any $t>3$; $b_{3,3}(4)=2$; $b_{4,3}(4)=2$; and $b_{t,3}(4)=3$, for any $t>4$.
\end{enumerate}
\end{theorem}
The rest of this paper is organized as follows. In Section 2 we prove Theorem \ref{Thm6.04} which gives certain hook length biases in ordinary partition functions. In
Section 3 we derive the generating functions of $b_{t,1}(n)$ for all $t\geq2$, namely we prove Theorem \ref{Thm6.1}. We also prove Theorem \ref{Thm6.2} which gives the generating functions for $b_{3,2}(n)$ and $b_{4,2}(n)$.  In Section 4 we prove the hook length biases in $2$-regular partitions, namely Theorems \ref{Thm6.5} and \ref{Thm6.6}. In Section 5 we prove Theorems \ref{Thm6.4} and \ref{Thm6.4a} which give values of $b_{t,k}(k)$ and $b_{t,k}(k+1)$. In Section 6 we make some concluding remarks and state some problems for future works.
\section{Hook length biases in ordinary partitions}\label{Section6.2}
In this section we prove Theorem \ref{Thm6.04}. We first discuss some prerequisites. We say that a series $S(q):=\sum_{n\geq0}s_nq^n$ is \emph{non-negative} if $s_n\geq0$, for all $n$. We use the notation $S(q)\succcurlyeq0$ to denote the non-negativity.
\par Let $a(n)$ denote the number of partitions of $n$ with smallest part at least 2. Define $a(0):=1$. Then the generating function for $a(n)$ is 
$$\sum_{n=0}^{\infty}a(n)q^n=\frac{1}{(q^2;q)_{\infty}}.$$
\begin{lemma}\label{Lemma6.3.1}
	For all positive integers $n\geq2$, $a(n+1)\geq a(n)$.
\end{lemma}
\begin{proof}
	For a positive integer $n\geq2$, let $\mathcal{P}^1(n)$ denote the set of all partitions of $n$ in which smallest part is at least 2. We define a map $\Psi_n:\mathcal{P}^1(n)\rightarrow\mathcal{P}^1(n+1)$ by
	$$\Psi_n((\lambda_1,\lambda_2,\ldots,\lambda_r))=(\lambda_1+1,\lambda_2,\ldots,\lambda_r).$$
	Clearly, $\Psi_n$ is a one-to-one map. Thus, $|\mathcal{P}^1(n)|\leq|\mathcal{P}^1(n+1)|$ and therefore $a(n)\leq a(n+1)$. 
\end{proof}
\begin{corollary}\label{Cor6.3.1}
	For any positive integer $k$, we have
	$$\frac{q^k-q^{k+1}}{(q^2;q)_{\infty}}=-q^{k+1}+H_k(q),$$
	where $H_k(q)\succcurlyeq0$.
\end{corollary}
\begin{proof}
	We have
	\begin{align*}
		\frac{q^k-q^{k+1}}{(q^2;q)_{\infty}}&=q^k\left(\sum_{n\geq0}(a(n)-a(n-1))q^n\right)\\
		&=q^k-q^{k+1}+q^{k+2}+\sum_{n\geq3}(a(n)-a(n-1))q^{n+k}.
	\end{align*}
	The proof follows from Lemma \ref{Lemma6.3.1}.
\end{proof}
We will repeatedly use the fact that for positive integers $\alpha,\ \beta$ with $\beta>\alpha+1$
\begin{align}\label{6.3.1}
	\frac{q^{\alpha}-q^{\beta}}{(q^2;q)_{\infty}}\succcurlyeq0.
\end{align}
Note that \eqref{6.3.1} is also an immediate consequence of Lemma \ref{Lemma6.3.1}.
\par We are now ready to prove Theorem \ref{Thm6.04}.
\begin{proof}[Proof of Theorem \ref{Thm6.04}] 
From \eqref{gen_fn_p(n)} and \eqref{6.1_Han}, we have
\begin{align}\label{6.3b}
\sum_{n=0}^{\infty}p_{(k)}(n)q^n=\sum_{r=0}^{\infty}p(r)q^r\left(\sum_{m=1}^{\infty}kq^{km} \right). 
\end{align}
If $k>r\geq 0$, then comparing the coefficients of $k+r$ on both sides of \eqref{6.3b} yields
\begin{align}\label{6.3c}
p_{(k)}(k+r)=p(r)\cdot k.
\end{align}
If $k\geq 2$, then taking $r=0,1$ in \eqref{6.3c}, we find that $p_{(k)}(k+1)-p_{(k+1)}(k+1)=-1$.
\par 
Next, to prove the biases, we define 
$$G_k(q):=\sum_{n=0}^{\infty}\left(p_{(k)}(n)-p_{(k+1)}(n)\right)q^n.$$
From \eqref{6.1_Han} we have
\begin{align}\label{6.3.5}
G_k(q)&=\frac{1}{(q;q)_{\infty}} \left(\frac{kq^k}{1-q^k}-\frac{(k+1)q^{k+1}}{1-q^{k+1}}\right)\nonumber\\
&=\frac{1}{(q;q)_{\infty}} \left(k\frac{q^k+q^{2k}+\cdots+q^{(k+1)k}}{1-q^{k^2+k}}-(k+1)\frac{q^{k+1}+q^{2(k+1)}+\cdots+q^{k(k+1)}}{1-q^{k^2+k}}\right)\nonumber\\
&=\frac{1}{(1-q^{k^2+k})(q;q)_{\infty}}\times\nonumber\\
&\hspace{.5cm}\left[k\left(q^k+q^{2k}+\cdots+q^{(k-1)k}-q^{k+1}-q^{2(k+1)}-\cdots-q^{(k-1)(k+1)}\right)\right.\nonumber\\
&\hspace{.5cm}\left.+\left(kq^{k^2}-q^{k+1}-q^{2(k+1)}-\cdots-q^{(k-1)(k+1)}-q^{k(k+1)}\right)\right]\nonumber\\
&=\frac{1}{(1-q^{k^2+k})(q^2;q)_{\infty}}\times\nonumber\\
&\hspace{.5cm}\left( k\left[\left(\frac{q^k-q^{k+1}}{1-q}\right)+\left(\frac{q^{2k}-q^{2(k+1)}}{1-q}\right)+\cdots+\left(\frac{q^{(k-1)k}-q^{(k-1)(k+1)}}{1-q}\right) \right]\right.  \nonumber\\
&\hspace{.5cm}+\left[\left(\frac{q^{k^2}-q^{k+1}}{1-q}\right)+\left(\frac{q^{k^2}-q^{2(k+1)}}{1-q}\right)+\cdots+\left(\frac{q^{k^2}-q^{(k-1)(k+1)}}{1-q}\right)\right]\nonumber\\
&\hspace{.5cm}\left.+\left(\frac{q^{k^2}-q^{k(k+1)}}{1-q}\right)\right)\nonumber\\
&=\frac{1}{(1-q^{k^2+k})(q^2;q)_{\infty}}\left( k\left[q^k+(q^{2k}+q^{2k+1})+(q^{3k}+q^{3k+1}+q^{3k+2})+\cdots+\right.\right.\nonumber\\
&\left.\hspace{.5cm} (q^{(k-1)k}+q^{(k-1)k+1}+\cdots+q^{(k-1)k+k-2})\right]-\left[(q^{k+1}+q^{k+2}+\cdots+q^{k^2-1})\right.\nonumber\\
&\left.\left. \hspace{.5cm}+(q^{2(k+1)}+q^{2(k+1)+1}+\cdots+q^{k^2-1})+\cdots+q^{k^2-1}\right]+\frac{q^{k^2}-q^{k^2+k}}{1-q}\right)\nonumber\\
&=\mathcal{K}_k(q)+\frac{\mathcal{L}_k(q)}{(1-q^{k^2+k})(q^2;q)_{\infty}},
\end{align}
where
\begin{align*}
\mathcal{K}_k(q):=\frac{q^k-q^{k+1}}{(1-q^{k^2+k})(q^2;q)_{\infty}}+\frac{q^{k^2}-q^{k^2+k}}{(1-q^{k^2+k})(q;q)_{\infty}}
\end{align*}
and
\begin{align}\label{new-eqn-01}
\mathcal{L}_k(q)&:=(k-1)q^k+k\left[(q^{2k}+q^{2k+1})+(q^{3k}+q^{3k+1}+q^{3k+2})+\cdots+\right.\nonumber\\
&\left.\hspace{.5cm} (q^{(k-1)k}+q^{(k-1)k+1}+\cdots+q^{(k-1)k+k-2})\right]\nonumber\\
&\hspace{.5cm} -\left[(q^{k+2}+q^{k+3}+\cdots+q^{k^2-1})+\right.\nonumber\\
&\left. \hspace{.5cm}(q^{2(k+1)}+q^{2(k+1)+1}+\cdots+q^{k^2-1})+\cdots+q^{k^2-1}\right].
\end{align}
Enumerating the terms in the second square bracket in the right hand side of \eqref{new-eqn-01}, we obtain
\begin{align}\label{6.3.2}
\mathcal{L}_k(q)&=(k-1)q^k+k\left[(q^{2k}+q^{2k+1})+(q^{3k}+q^{3k+1}+q^{3k+2})+\cdots+\right.\nonumber\\
&\left.\hspace{.5cm}(q^{(k-1)k}+q^{(k-1)k+1}+\cdots+q^{(k-1)k+k-2})\right]\nonumber\\
&\hspace{.5cm}-\left[(q^{k+2}+q^{k+3}+\cdots+q^{2k+1})+2(q^{2k+2}+q^{2k+3}+\cdots+q^{3k+2})+\cdots+\right.\nonumber\\
&\hspace{.5cm}(k-2)(q^{(k-2)(k+1)}+q^{(k-2)(k+1)+1}+\cdots+q^{(k-1)(k+1)-1})\nonumber\\
&\left.\hspace{.5cm}+(k-1)q^{(k-1)(k+1)}\right]\nonumber\\
&=\left[(k-1)q^k-(q^{k+2}+q^{k+3}+\cdots+q^{2k-1})\right]+\nonumber\\
&\hspace{.5cm}\left[(k-1)(q^{2k}+q^{2k+1})-2(q^{2k+2}+q^{2k+3}+\cdots+q^{3k-1})\right]+\nonumber\\
&\hspace{.5cm}\left[(k-2)(q^{3k}+q^{3k+1}+q^{3k+2})-3(q^{3k+3}+q^{3k+4}+\cdots+q^{4k-1})\right]+\cdots+\nonumber\\
&\hspace{.5cm}\left[2(q^{(k-1)k}+q^{(k-1)k+1}+\cdots+q^{(k-1)k+(k-2)})-(k-1)(q^{(k-1)(k-2)})\right].
\end{align}
Using \eqref{6.3.1} repeatedly on the right hand side of \eqref{6.3.2}, we obtain that
\begin{align}\label{6.3.3}
\frac{\mathcal{L}_k(q)}{(q^2;q)_{\infty}}\succcurlyeq0.
\end{align}
Also,
\begin{align*}
\mathcal{K}_k(q)=&\frac{q^k-q^{k+1}}{(q^2;q)_{\infty}}\left(1+\frac{q^{k^2+k}}{1-q^{k^2+k}}\right) +\frac{q^{k^2}-q^{k^2+k}}{(1-q^{k^2+k})(q;q)_{\infty}}\\
=&\frac{q^k-q^{k+1}}{(q^2;q)_{\infty}}+\frac{1}{(1-q^{k^2+k})(q^2;q)_{\infty}}\times\\
&\left(q^{k^2+2k}-q^{k^2+2k+1}+(q^{k^2}+q^{k^2+1}+\cdots+q^{k^2+k-1})\right).
\end{align*}
From Corollary \ref{Cor6.3.1} and using \eqref{6.3.1} on $\frac{q^{k^2}-q^{k^2+2k+1}}{(q^2;q)_{\infty}}$, we obtain that
\begin{align}\label{6.3.4}
\mathcal{K}_k(q)=-q^{k+1}+J_k(q),
\end{align}
where $J_k(q)\succcurlyeq0$.
Invoking \eqref{6.3.3} and \eqref{6.3.4} in \eqref{6.3.5} yields
\begin{align*}
G_k(q)=-q^{k+1}+M_k(q),
\end{align*}
where $M_k(q)\succcurlyeq0$. This completes the proof.
\end{proof}
\section{Generating functions for $b_{t,k}(n)$}\label{Section6.1}
In this section, we derive the generating functions of $b_{t,1}(n)$ for all $t\geq2$. We also derive the generating functions for $b_{3,2}(n)$ and $b_{4,2}(n)$. 
First, we introduce some notations that would be useful in this section. For any part $s$ of a partition $\lambda$, we denote by $m_{\lambda}(s)$ the multiplicity of $s$ in $\lambda$, i.e., the number of times $s$ appears as a part in $\lambda$. 
For a positive integer $j$, let $\ell_{>j}(\lambda)$ denote the number of part sizes greater than $j$ in the partition $\lambda$. By $\ell_1(\lambda)$ (respectively $\ell_2(\lambda)$) 
we denote the number of parts $\lambda_i$ of $\lambda$ with $\lambda_i-\lambda_{i+1}=1$ (respectively $\lambda_i-\lambda_{i+1}=2$), where we take $\lambda_i=0$ if $i>\ell(\lambda)$. 
Let $\delta_P$ denote the Kronecker delta symbol, which is equal to 1 if property $P$ is true and 0 otherwise.
\par Let $b_{t,k}(m,n)$ denote the number of $t$-regular partitions of $n$ with $m$ hooks of length $k$. Then,
\begin{align}\label{6.1.1}
b_{t,k}(n)=\sum_{m=0}^{\infty}mb_{t,k}(m,n).
\end{align}
\begin{proof}[Proof of Theorem \ref{Thm6.1}]
In the Young diagram of a partition $\lambda$, there is a box with hook length 1 if and only if that box is at the end of a row with no box directly below it. 
Therefore, the number of hooks of length 1 in a partition $\lambda$ is same as the number of distinct parts in $\lambda$. Thus, 
\begin{align*}
\mathcal{B}_{t,1}(z;q):=\sum_{n,m\geq0}b_{t,1}(m,n)z^mq^n=\prod_{\substack{n\geq1\\ t\nmid n}}\left(1+\frac{zq^n}{1-q^n}\right).
\end{align*}
From \eqref{6.1.1}, we have
\begin{align}\label{6.1.2}
\sum_{n=0}^{\infty}b_{t,1}(n)q^n=\frac{\partial}{\partial z}\bigg\vert_{z=1}\mathcal{B}_{t,1}(z;q).
\end{align}
Using logarithmic differentiation, \eqref{6.1.2} yields \eqref{6.0.3}.
\end{proof}
\begin{proof}[Proof of Theorem \ref{Thm6.2}]
Let $\lambda$ be a 3-regular partition. We observe that the number of hooks of length 2 in $\lambda$ is equal to the number of different parts of $\lambda$ that occur at 
least twice plus the number of parts $\lambda_i$ of $\lambda$ for which $\lambda_i-\lambda_{i+1}\geq2$, where we take $\lambda_i=0$ if $i>\ell(\lambda)$. For example, see Figure \ref{Figure6.2}.
\begin{figure}[h]
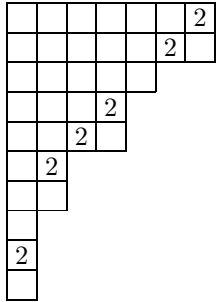

\centering
\[\young(~~~~~~2,~~~~~2~,~~~~~,~~~2,~~2~,~2,~~,~,2,~)\]
\caption{The hooks of length 2 in the Young diagram of the partition\label{Figure6.2} $(7,7,5,4,4,2,2,1,1,1)$}
\end{figure}
 The number of parts $\lambda_i$ of $\lambda$ for which $\lambda_i-\lambda_{i+1}\geq2$ is nothing but 
\begin{align*}
\ell_{>1}(\lambda)-\ell_1(\lambda)=\sum_{s\geq2}\delta_{m_{\lambda}(s)\geq1}-\ell_1(\lambda).
\end{align*}
Thus, $b_{3,2}(m,n)$ is equal to the number of 3-regular partitions $\lambda$ of $n$ such that
\begin{align*}
\delta_{m_{\lambda}(1)\geq2}+\sum_{s\geq2}\left( \delta_{m_{\lambda}(s)\geq1}+\delta_{m_{\lambda}(s)\geq2}\right)-\ell_1(\lambda)=m.
\end{align*}
Next, define $u(r,n)$ to be the number of 3-regular partitions $\lambda$ of $n$ such that
\begin{align*}
\delta_{m_{\lambda}(1)\geq2}+\sum_{s\geq2}\left( \delta_{m_{\lambda}(s)\geq1}+\delta_{m_{\lambda}(s)\geq2}\right)=r
\end{align*}
and $v(r,n)$ to be the number of 3-regular partitions $\lambda$ of $n$ such that $\ell_1(\lambda)=r$.
Notice that
\begin{align}\label{6.1.3}
b_{3,2}(n)=\sum_{m=0}^{\infty}mb_{3,2}(m,n)=\sum_{r=0}^{\infty}ru(r,n)-\sum_{r=0}^{\infty}rv(r,n).
\end{align}
If we define $\mathcal{B}^{1}_{3,2}(z;q):=\displaystyle\sum_{n,r\geq0}u(r,n)z^rq^n$ and $\mathcal{B}^{2}_{3,2}(z;q):=\displaystyle\sum_{n,r\geq0}v(r,n)z^rq^n$, then using \eqref{6.1.3}, we arrive at
\begin{align}\label{6.1.4}
\sum_{n=0}^{\infty}b_{3,2}(n)q^n=\frac{\partial}{\partial z}\bigg\vert_{z=1}\mathcal{B}^{1}_{3,2}(z;q)-\frac{\partial}{\partial z}\bigg\vert_{z=1}\mathcal{B}^{2}_{3,2}(z;q).
\end{align}
From the definition of $u(r,n)$, we have
\begin{align*}
\mathcal{B}^{1}_{3,2}(z;q)=\left(1+q+\frac{zq^2}{1-q}\right)\prod_{\substack{n\geq2\\ 3\nmid n}}\left(1+zq^n+\frac{z^2q^{2n}}{1-q^n}\right).  
\end{align*}
Now, $v(r,n)$ counts the number of 3-regular partitions $\lambda$ of $n$ with $r$ number of parts $\lambda_i$ such that $\lambda_i-\lambda_{i+1}=1$. 
If $\lambda$ is a 3-regular partition, then the parts $\lambda_i$ of $\lambda$ for which $\lambda_i-\lambda_{i+1}=1$ are precisely those which are congruent to 2 modulo 3 and part $\lambda_i-1$, which is congruent to 1 modulo 3, is also a part of $\lambda$. Therefore,
\begin{align*}
\mathcal{B}^{2}_{3,2}(z;q)=\prod_{n\geq0}\left(1+\frac{q^{3n+1}}{1-q^{3n+1}}+\frac{q^{3n+2}}{1-q^{3n+2}}+\frac{zq^{3n+1}\cdot q^{3n+2}}{(1-q^{3n+1})(1-q^{3n+2})}\right).
\end{align*} 
By logarithmic differentiation, we obtain
\begin{align}
\frac{\partial}{\partial z}\bigg\vert_{z=1}\mathcal{B}^{1}_{3,2}(z;q)&=\prod_{\substack{n\geq1\\ 3\nmid n}}\frac{1}{1-q^n}\left(q^2+\sum_{\substack{n\geq2\\ 3\nmid n}}\left(q^n+q^{2n}\right) \right)\nonumber\\
&=\frac{(q^3;q^3)_{\infty}}{(q;q)_{\infty}}\left(\frac{q^2}{1-q}+\frac{q^2}{1-q^2}-\frac{q^3}{1-q^3}-\frac{q^6}{1-q^6}\right) \label{6.1.5}
\end{align}
and
\begin{align}
\frac{\partial}{\partial z}\bigg\vert_{z=1}\mathcal{B}^{2}_{3,2}(z;q)&=\prod_{n\geq0}\frac{1}{(1-q^{3n+1})(1-q^{3n+2})}\left(\sum_{n\geq0}q^{3n+1}\cdot q^{3n+2} \right)\nonumber\\
&=\frac{(q^3;q^3)_{\infty}}{(q;q)_{\infty}}\left(\frac{q^3}{1-q^6}\right).\label{6.1.6}
\end{align}
Finally, from \eqref{6.1.4}, \eqref{6.1.5}, and \eqref{6.1.6}, we obtain that
\begin{align}\label{eqn-new-02}
\sum_{n=0}^{\infty}b_{3,2}(n)q^n&=\frac{(q^3;q^3)_{\infty}}{(q;q)_{\infty}}\left(\frac{q^2}{1-q}+\frac{q^2}{1-q^2}-\frac{q^3}{1-q^3}-\frac{q^6}{1-q^6}-\frac{q^3}{1-q^6}\right).
\end{align}
We readily obtain \eqref{6.0.4} from \eqref{eqn-new-02}.
\par Next, we prove \eqref{6.0.5}, the generating function for $b_{4,2}(n)$, on the similar steps of the above proof of \eqref{6.0.4}. In this case
\begin{align}\label{6.1.7}
b_{4,2}(n)=\sum_{m=0}^{\infty}mb_{4,2}(m,n)=\sum_{r=0}^{\infty}r\bar{u}(r,n)-\sum_{r=0}^{\infty}r\bar{v}(r,n),
\end{align}
where
$\bar{u}(r,n)$ is the number of 4-regular partitions $\lambda$ of $n$ such that
\begin{align*}
\delta_{m_{\lambda}(1)\geq2}+\sum_{s\geq2}\left( \delta_{m_{\lambda}(s)\geq1}+\delta_{m_{\lambda}(s)\geq2}\right)=r
\end{align*}
and $\bar{v}(r,n)$ to be the number of 4-regular partitions $\lambda$ of $n$ such that $\ell_1(\lambda)=r$.
If we define $\mathcal{B}^{1}_{4,2}(z;q):=\displaystyle\sum_{n,r\geq0}\bar{u}(r,n)z^rq^n$ and $\mathcal{B}^{2}_{4,2}(z;q):=\displaystyle\sum_{n,r\geq0}\bar{v}(r,n)z^rq^n$, then using \eqref{6.1.7}, we arrive at
\begin{align}\label{6.1.8}
\sum_{n=0}^{\infty}b_{4,2}(n)q^n=\frac{\partial}{\partial z}\bigg\vert_{z=1}\mathcal{B}^{1}_{4,2}(z;q)-\frac{\partial}{\partial z}\bigg\vert_{z=1}\mathcal{B}^{2}_{4,2}(z;q).
\end{align}
From the definition of $\bar{u}(r,n)$ we have
\begin{align*}
\mathcal{B}^{1}_{4,2}(z;q)=\left(1+q+\frac{zq^2}{1-q}\right)\prod_{\substack{n\geq2\\ 4\nmid n}}\left(1+zq^n+\frac{z^2q^{2n}}{1-q^n}\right).  
\end{align*}
Now, $\bar{v}(r,n)$ counts the number of 4-regular partitions $\lambda$ of $n$ with $r$ number of parts $\lambda_i$ such that $\lambda_i-\lambda_{i-1}=1$. If $\lambda$ is a 4-regular partition, 
then a part $\lambda_i$ of $\lambda$ for which $\lambda_i-\lambda_{i-1}=1$ can appear in these cases only:
\begin{enumerate}
\item $\lambda_i\equiv3\pmod4$ and part $\lambda_i-1\equiv2\pmod4$ is also a part of $\lambda$, but $\lambda_i-2$ is not a part of $\lambda$;
\item $\lambda_i\equiv2\pmod4$ and part $\lambda_i-1\equiv1\pmod4$ is also a part of $\lambda$, but $\lambda_i+1$ is not a part of $\lambda$;
\item $\lambda_i\equiv3\pmod4$ and parts $\lambda_i-1\equiv2\pmod4$ and $\lambda_i-2$ are also parts of $\lambda$.
\end{enumerate}
Therefore,
\begin{align*}
\mathcal{B}^{2}_{4,2}(z;q)=\prod_{n\geq0}&\left(1+\sum_{i=1}^{3}\frac{q^{4n+i}}{1-q^{4n+i}}+\frac{zq^{4n+2}}{1-q^{4n+2}}\left( \frac{q^{4n+1}}{1-q^{4n+1}}+\frac{q^{4n+3}}{1-q^{4n+3}}\right)\right. \\
&\left.+\frac{z^2q^{4n+1}\cdot q^{4n+2}\cdot q^{4n+3}}{(1-q^{4n+1})(1-q^{4n+2})(1-q^{4n+2})}\right).
\end{align*}
By logarithmic differentiation, we obtain
\begin{align}
\frac{\partial}{\partial z}\bigg\vert_{z=1}\mathcal{B}^{1}_{4,2}(z;q)=\prod_{\substack{n\geq1\\ 4\nmid n}}\frac{1}{1-q^n}\left(q^2+\sum_{\substack{n\geq2\\ 4\nmid n}}\left(q^n+q^{2n}\right) \right)\label{6.1.9}
\end{align}
and
\begin{align}
\frac{\partial}{\partial z}\bigg\vert_{z=1}\mathcal{B}^{2}_{4,2}(z;q)=&\prod_{\substack{n\geq1\\ 4\nmid n}}\frac{1}{1-q^n}\left( \prod_{n\geq0}(1-q^{4n+1}q^{4n+3})\right)\times\nonumber\\ &\left(\sum_{n\geq0}\frac{q^{4n+2}(q^{4n+1}+q^{4n+3})}{1-q^{4n+1}q^{4n+3}} \right).\label{6.1.10}
\end{align}
Finally, \eqref{6.1.8}, \eqref{6.1.9}, and \eqref{6.1.10} yield \eqref{6.0.5}.
\end{proof}
\begin{remark}
The method used in Theorem \ref{Thm6.2} to obtain the generating functions for $b_{3,2}(n)$ and $b_{4,2}(n)$ can be extended to derive the generating function for $b_{t,2}(n)$, for further values of $t$. But as $t$ increases, the generating functions become more complex.
\end{remark}
\section{Hook length biases in 2-regular partitions}\label{Section6.3}
In this section, we give proofs of Theorems \ref{Thm6.5} and \ref{Thm6.6}. First, we recall another form of representation of a partition given by $$\lambda=(\lambda^{m_1}_1,\lambda^{m_2}_2,\ldots,\lambda^{m_r}_r),$$ 
where $m_i$ is the multiplicity of the part $\lambda_i$ and $\lambda_1>\lambda_2>\cdots>\lambda_r$.
\begin{proof}[Proof of Theorem \ref{Thm6.5}]
For $n>4$, let $\mathcal{B}_2(n)$ denote the set of all 2-regular or odd partitions of $n$. Also, let $\mathcal{R}_n\subset\mathcal{B}_2(n)$ be the set of partitions having at least one part of multiplicity more than 1; $\mathcal{S}_n\subset\mathcal{B}_2(n)$ be 
the set of distinct partitions with smallest part 1; and $\mathcal{T}_n\subset\mathcal{B}_2(n)$ be the set of distinct partitions with smallest part greater than 1. Then $\mathcal{B}_2(n)=\mathcal{R}_n\cup\mathcal{S}_n\cup\mathcal{T}_n$. Given a partition $\lambda\in \mathcal{B}_2(n)$, let $h_k(\lambda)$ denote the number of hooks of length $k$ in the partition $\lambda$.
\par \underline{Case 1:} Let $\lambda\in\mathcal{R}_n$. If a part $\lambda_i$ of $\lambda$ occurs at least twice then the rows corresponding to $\lambda_i$ in the Young diagram of $\lambda$ will have two hooks of length 2 and one hook of length 1. For every other part of $\lambda$ 
which is greater than one, there is one hook of length 1 and at least one hook of length 2. For part of size 1 (if present as a part in $\lambda$), there is a hook of length 1 and there may or may not be a hook of length 2. Considering all the cases, we conclude that the number of hooks of length 1 in $\lambda$ is not more than the number of hooks of length 2 in $\lambda$. Hence, $h_2(\lambda)\geq h_1(\lambda)$ for all $\lambda\in\mathcal{R}_n$. 
\par \underline{Case 2:} Let $\lambda\in\mathcal{S}_n$. Clearly, the number of hooks of length 1 in $\lambda$ is exactly one more than the number of hooks of length 2 in $\lambda$. That is, $h_1(\lambda)=1+h_2(\lambda)$ for all $\lambda\in\mathcal{S}_n$.
\par \underline{Case 3:} Let $\lambda\in\mathcal{T}_n$. Then, every part of $\lambda$ corresponds to exactly one hook of length 1 and one hook of length 2 in the Young diagram of $\lambda$. Therefore, in this case the number of hooks of length 1 in $\lambda$ is same as the number of hooks of length 2 in $\lambda$. Thus, $h_1(\lambda)=h_2(\lambda)$ for all $\lambda\in\mathcal{T}_n$.
\par In order to establish the required hook lengths bias, we next show that the extra hook of length 1 coming from Case 2 will be adjusted in Case 1. For this, we define a map $\Phi_n:\mathcal{S}_n\rightarrow\mathcal{R}_n$ by
$$\Phi_n\left( (\lambda_1,\lambda_2,\ldots,\lambda_r,1)\right) :=(\lambda^2_2,\lambda_3,\ldots,\lambda_r,1^{\lambda_1-\lambda_2+1})$$
if $r\geq2$ and if $r=1$ (which is possible only when $n$ is even),
$$\Phi_n\left( (\lambda_1=n-1,1)\right):=\left\{
\begin{array}{lll}
((\frac{n-2}{2})^2,1^2) & \text{if}\ n\equiv0\pmod4;\vspace{.12cm}\\
((\frac{n}{2})^2) & \text{if}\ n\equiv2\pmod4.
\end{array}
\right.$$
Clearly, $\Phi_n$ is a one-to-one map.  For any $\lambda\in\Phi_n(\mathcal{S}_n)$, 
the number of hooks of length 1 in $\lambda$ is exactly one less than the number of hooks of length 2 in $\lambda$. Thus, $h_1(\lambda)=h_2(\lambda)-1$ for all $\lambda\in\Phi_n(\mathcal{S}_n)$. 
\par Combining all the cases and the fact that $\Phi_n$ is one-to-one, we have
\begin{align*}
b_{2,2}(n)&=\sum_{\lambda\in \mathcal{B}_2(n)}h_2(\lambda)\\
&=\sum_{\lambda\in \mathcal{R}_n}h_2(\lambda)+\sum_{\lambda\in \mathcal{S}_n}h_2(\lambda)+\sum_{\lambda\in \mathcal{T}_n}h_2(\lambda)\\
&=\sum_{\lambda\in \mathcal{R}_n\setminus \Phi_n(\mathcal{S}_n)}h_2(\lambda)+\sum_{\lambda\in \Phi_n(\mathcal{S}_n)}h_2(\lambda)+\sum_{\lambda\in \mathcal{S}_n}h_2(\lambda)+\sum_{\lambda\in \mathcal{T}_n}h_2(\lambda)\\
&\geq \sum_{\lambda\in \mathcal{R}_n\setminus \Phi_n(\mathcal{S}_n)}h_1(\lambda)+\sum_{\lambda\in \Phi_n(\mathcal{S}_n)}(h_1(\lambda)+1)+\sum_{\lambda\in \mathcal{S}_n}(h_1(\lambda)-1)+\sum_{\lambda\in \mathcal{T}_n}h_1(\lambda)\\
&=\sum_{\lambda\in \mathcal{B}_2(n)}h_1(\lambda)\\
&=b_{2,1}(n).
\end{align*}
This completes the proof of the theorem.
\end{proof}
\begin{proof}[Proof of Theorem \ref{Thm6.6}]
Using well-known Euler's identity \cite[(1.2.5)]{Andrews_1998}, $1/(q;q^2)_{\infty}=(-q;q)_{\infty}$ and straightforward $q$-series manipulations, we rewrite \eqref{6.0.1} as
\begin{align}
\sum_{n=0}^{\infty}b_{2,2}(n)q^n&=(-q;q)_{\infty}\left(q^2+\frac{q^3}{1-q^2}+\frac{q^6}{1-q^4} \right)\nonumber\\
&=(-q;q)_{\infty}\left(\frac{q^3}{1-q^2}+\frac{q^2}{1-q^4} \right)\nonumber\\
&=(-q^3;q)_{\infty}\left(\frac{q^2+2q^3+q^4+q^5+q^6}{1-q^2} \right).\label{6.2.1}
\end{align}
Similarly, we rewrite \eqref{6.0.2} to obtain
\begin{align}\label{6.2.2}
\sum_{n=0}^{\infty}b_{2,3}(n)q^n&=(-q^3;q)_{\infty}\left(\frac{q^3+2q^6+q^7}{1-q^2}\right) +(-q;q)_{\infty}\left(\frac{q^3}{1-q^6}\right).
\end{align}
From \eqref{6.2.1} and \eqref{6.2.2}, we get
\begin{align}\label{6.2.3.}
\sum_{n=0}^{\infty}(b_{2,2}(n)-b_{2,3}(n))q^n&=(-q^3;q)_{\infty}\left(\frac{q^2+q^3+q^4+q^5-q^6-q^7}{1-q^2}\right)\nonumber\\
&-(-q;q)_{\infty}\left(\frac{q^3}{1-q^6}\right)\nonumber\\
&=(-q^3;q)_{\infty}\left(\left(\frac{q^2-q^6}{1-q^2}\right)+\left(\frac{q^3-q^7}{1-q^2}\right)+\left(\frac{q^4+q^5}{1-q^2}\right) \right)\nonumber\\
&-(-q;q)_{\infty}\left(\frac{q^3}{1-q^6}\right)\nonumber\\
&=(-q^3;q)_{\infty}\left(q^2+q^3+q^4+q^5+\frac{q^4+q^5}{1-q^2} \right)\nonumber\\
&-(-q;q)_{\infty}\left(\frac{q^3}{1-q^6}\right)\nonumber\\
&=(-q^3;q)_{\infty}\left(\frac{q^2+q^5-q^6-q^9-q^{10}-q^{11}+q^{12}+q^{13}}{(1-q^2)(1-q^6)}\right)\nonumber\\
&=(-q^3;q)_{\infty}\left(\frac{(1-q^2)(q^2+q^4+q^5+q^7-q^{10}-q^{11})}{(1-q^2)(1-q^6)}\right)\nonumber\\
&=\frac{(-q^4;q)_{\infty}}{1-q^3}\left(q^2+q^4+q^5+q^7-q^{10}-q^{11}\right)\nonumber\\
&=(-q^4;q)_{\infty}\left(\frac{q^2+q^7}{1-q^3}+\left(\frac{q^4-q^{10}}{1-q^3}\right)+\left( \frac{q^5-q^{11}}{1-q^3}\right)  \right)\nonumber\\
&=(-q^4;q)_{\infty}\left(\frac{q^2+q^7}{1-q^3}+q^4+q^5+q^7+q^8\right).
\end{align}
Clearly, all the coefficients of the expression in \eqref{6.2.3.}, when expanded as a $q$-series, are non-negative. This completes the proof of Theorem \ref{Thm6.6}.
\end{proof}
\section{Values of $b_{t,k}(k)$ and $b_{t,k}(k+1)$}
In this section, we prove Theorems \ref{Thm6.4} and \ref{Thm6.4a}. 
\begin{proof}[Proof of Theorem \ref{Thm6.4}]
Let $k=2r$ and $\lambda$ be a 2-regular partition of $k$ with a hook of length $k$. Then there are $k$ boxes in the Young diagram of $\lambda$ and the only box of hook of length $k$ is at the top left corner of the Young diagram. 
The 2-regular partitions of $k$ with hook of length $k$ are precisely: $(2r-1,1),\ (2r-3,1^3),\ldots,\ (3,1^{2r-3}),\ (1^{2r})$, which are $r=k/2$ in number. If $k=2r+1$, then the 2-regular partitions of $k$ with hook of length $k$ are precisely: $(2r+1),\ (2r-1,1^2),\ldots,\ (1^{2r+1})$, 
which are $r+1=(k+1)/2$ in number. This proves \eqref{6.0.41}.
\par Let $k=2r$ and $\lambda$ be a 2-regular partition of $k+1$ with a hook of length $k$. Then, the only possible partitions of $k+1$ with a hook of length $k$ are $(2r+1)$ and $(1^{2r+1})$. If $k=2r+1$, then there is only one possible 2-regular partition of $k+1$ with a hook of length $k$, 
which is $(1^{2r+1})$. This gives \eqref{6.0.42}.
\end{proof}
\begin{proof}[Proof of Theorem \ref{Thm6.4a}]
The Young diagram of a partition $\lambda$ of $k$ can have at most one hook of length $k$. The shape of the Young diagram of $\lambda$ with the hook of length $k$ is that of either an inverted L-shape or a vertical strip or a horizontal strip. Let $k=t\ell+r$, 
where $0\leq r\leq t-1$. If $r=0$ then the only $t$-regular partitions of $k$ with the hook of length $k$ are: $(t\ell-j,1^j)$, where $1\leq j\leq t\ell-1$ and $t\nmid j$, which are $\frac{(t-1)k}{t}$ in number. 
If $0<r\leq t-1$ then the corresponding partitions are: $(t\ell+r-j)$, where $0\leq j\leq t\ell+r-1$ with $t\nmid(r-j)$, which are $\frac{(t-1)k+r}{t}$ in number. This proves \eqref{6.0.41a}.
\par Next, the Young diagram of a partition $\lambda$ of $k+1$ can have at most one hook of length $k$. In this case, the shape of the Young diagram is that of either a vertical strip or a horizontal strip or an 
inverted L-shape with the second part being of size 2. Let $k=t\ell+r$, where $0\leq r\leq t-1$. If $r=0$ then the only $t$-regular partitions of $k+1$ with the hook of length $k$ are: $(t\ell-j,2,1^{j-1})$, 
where $1\leq j\leq t\ell-2$ and $t\nmid j$; $(t\ell+1)$ (horizontal strip); $(1^{t\ell+1})$ (vertical strip), which are $((t-1)\ell-1)+2=\frac{(t-1)k+t}{t}$ in number. If $0<r<t-1$ then the corresponding partitions are: $(t\ell+r-j,2,1^{j-1})$, 
where $1\leq j\leq t\ell+r-1$ with $t\nmid(r-j)$; $(t\ell+r+1)$ (horizontal strip); $(1^{t\ell+r+1})$ (vertical strip), which are $((t-1)\ell+r-2)+2=\frac{(t-1)k+r}{t}$ in number. If $r=t-1$ then the corresponding partitions are: $(t\ell+t-1-j,2,1^{j-1})$, 
where $1\leq j\leq t\ell+r-1$ with $t\nmid(j+1)$; (no horizontal strip in this case); $(1^{t\ell+t})$ (vertical strip), which are $((t-1)\ell+r-2)+2=\frac{(t-1)k+r}{t}$ in number. This proves \eqref{6.0.42a}.
\end{proof}
\section{Concluding Remarks}
Using \texttt{SageMath} and the generating functions of $b_{3,1}(n)$ and $b_{3,2}(n)$ derived in \eqref{6.0.3} and \eqref{6.0.4}, we find that: 
\begin{align*}
\sum_{n=0}^{\infty}&(b_{3,2}(n)-b_{3,1}(n))q^n=-q-2q^3-3q^5-q^6-4q^7-2q^8-6q^9-3q^{10}-9q^{11}\\
&-4q^{12}-12q^{13}-6q^{14}-15q^{15}-8q^{16}-19q^{17}-9q^{18}-22q^{19}-9q^{20}-24q^{21}\\
&-7q^{22}-23q^{23}-17q^{25}+14q^{26}-2q^{27}+40q^{28}+27q^{29}+84q^{30}+77q^{31}\\
&+156q^{32}+159q^{33}+267q^{34}+289q^{35}+435q^{36}+486q^{37}+685q^{38}+778q^{39}\\
&+1049q^{40}+1202q^{41}+1570q^{42}+1809q^{43}+2307q^{44}+2665q^{45}+3335q^{46}\\
&+3859q^{47}+4756q^{48}+5504q^{49}+6701q^{50}+7750q^{51}+9341q^{52}+10791q^{53}\\
&+12895q^{54}+14877q^{55}+17646q^{56}+20326q^{57}+23956q^{58}+27548q^{59}\\
&+32286q^{60}+37059q^{61}+43219q^{62}+49518q^{63}+57494q^{64}+65749q^{65}\\
&+76038q^{66}+86796q^{67}+100016q^{68}+113959q^{69}+130885q^{70}+\cdots.
\end{align*}
In view of above, we conjecture the following bias.
\begin{conj}\label{conj6.5.1}
For all $n\geq28$, $b_{3,2}(n)\geq b_{3,1}(n)$.
\end{conj}
This study gives rise to many interesting questions. We see in Theorem \ref{Thm6.04} that the number of hooks of length $k$ is at least the number of hooks of length $k+1$ in all partitions of $n$, for all $n$ except $n=k+1$. 
But in Theorem \ref{Thm6.5}, we find that this bias reverses for $k=1$ in the case of 2-regular partitions. 
In the direction of Conjecture \ref{conj6.5.1} and following Theorem \ref{Thm6.5}, it would be interesting to study the relationship between $b_{t,1}(n)$ and $b_{t,2}(n)$, for $t\geq3$. Also, we see in Theorem \ref{Thm6.6} 
that the bias for 2-regular partitions remains the same as in the ordinary partitions for $k=2$. Apart from Conjecture \ref{conj6.1}, it would also be interesting to study the relationship between $b_{t,k}(n)$ and $b_{t,k+1}(n)$, for $t\geq3$ and $k\geq2$. 
\par It is evident from \eqref{6.1_Han} that $p_{(k)}(n)\equiv0\pmod{k}$, for every $k$. It would also be natural to examine the existence of such congruences for $b_{t,k}(n)$. In light of \eqref{6.3c}, it would be interesting to look for the corresponding expression for $b_{t,k}(n)$. 
\section{Acknowledgement}
The authors thank the referee for many valuable comments. The second author gratefully acknowledge the Department of Science and Technology, government of India, for the Core Research Grant (CRG/2021/00314) of SERB.


\begin{thebibliography}{999}
\bibitem{Ono_Singh}
T. Amdeberhan, G. E. Andrews, K. Ono, and A. Singh, {\it Hook lengths in self-conjugate partitions}, preprint. arXiv:2312.02933	

\bibitem{Andrews_1998}
G. E. Andrews, {\it The Theory of Partitions}, Cambridge Mathematical Library, Cambridge University Press, Cambridge, 1998. Reprint of the 1976 original.	
	
\bibitem{Ballantine_2023}
C. Ballantine, H. E. Burson, W. Craig, A. Folsom, and B. Wen, {\it Hook length biases and general linear partition inequalities}, Res. Math. Sci. 10, 41 (2023).

\bibitem{Craig}
W. Craig, M. L. Dawsey, and G.-N. Han, {\it Inequalities and asymptotics for hook numbers in restricted partitions}, preprint. arXiv:2311.15013

\bibitem{Garvan_1990}
F. Garvan, D. Kim, and D. Stanton, {\it Cranks and t-cores}, Invent. Math. 101 (1990), 1--18.

\bibitem{Han_2010}
G.-N. Han, {\it The Nekrasov-Okounkov hook length formula: refinement, elementary proof, extension and applications}, Ann. Inst. Fourier 60.1 (2010), 1--29. 

\bibitem{Han_2016}
G-N. Han and H. Xiong, {\it New hook-content formulas for strict partitions}, In: 28th International Conference on Formal Power Series and Algebraic Combinatorics (FPSAC 2016), Discrete Math. Theor. Comput. Sci. Proc. (2016), 635--645.

\bibitem{Han_2017}
G-N. Han and H. Xiong, {\it New hook-content formulas for strict partitions}, J. Algebraic Combin. 45(4) (2017), 1001--1019. 

\bibitem{James}
G. James and A. Kerber, {\it The representation theory of the symmetric group}, Ency. of Math. and its Appl., vol. 16, Addison-Wesley Publishing Co., Reading, Mass., 1981.

\bibitem{Littlewood}
D. E. Littlewood, {\it Modular representations of symmetric groups}, Proc. Roy. Soc. London, Ser. A 209 (1951), 333--353.

\bibitem{Nekrasov}
N. A. Nekrasov and A. Okounkov, {\it Seiberg-Witten theory and random partitions, in The unity of mathematics}, Prog. Math., Birkh{\"a}user Boston (2006), vol. 244, 525--596.

\bibitem{Petreolle}
M. P{\'e}tr{\'e}olle, {\it Quelques d{\'e}veloppements combinatoires autour des groupes de Coxeter et des partitions d'entiers}, Theses, Universit{\'e} Claude Bernard - Lyon I, November 2015.

\bibitem{Stanley}
R. P. Stanley, {\it Enumerative combinatorics}. Vol. 2, volume 62 of Cambridge Studies in Advanced Mathematics, Cambridge University Press, Cambridge, (1999).
\end{thebibliography}
\end{document}